\documentclass[11pt,reqno]{amsart}

\usepackage[letterpaper,margin=1.05in]{geometry}
\usepackage{amsmath,amssymb,amsthm,mathtools}
\usepackage{mathrsfs,bm}
\usepackage{enumitem}
\usepackage{microtype}
\usepackage[colorlinks=true, linkcolor=blue,urlcolor=blue]{hyperref}
\hypersetup{
  hidelinks,
  pdftitle={Deformed Convolution, Cumulant Transforms, and Semigroup Generators in Hurwitz Series Rings},
  pdfauthor={Morteza Ahmadi},
  pdfsubject={Hurwitz series rings, deformed convolution, cumulants, and semigroup generators},
  pdfkeywords={Hurwitz series ring, deformed convolution, finite free cumulant, infinitesimal generator, Levy process}
}
\allowdisplaybreaks

\newtheorem{theorem}{Theorem}[section]
\newtheorem{proposition}[theorem]{Proposition}
\newtheorem{lemma}[theorem]{Lemma}
\newtheorem{corollary}[theorem]{Corollary}
\theoremstyle{definition}
\newtheorem{definition}[theorem]{Definition}
\newtheorem{example}[theorem]{Example}
\theoremstyle{remark}

\newcommand{\N}{\mathbb N}
\newcommand{\Z}{\mathbb Z}
\newcommand{\Q}{\mathbb Q}
\newcommand{\R}{\mathbb R}
\newcommand{\C}{\mathbb C}
\newcommand{\fall}[2]{(#1)_{\underline{#2}}}
\newcommand{\rise}[2]{(#1)_{\overline{#2}}}
\newcommand{\hconv}[1]{\mathbin{\boxplus^{H}_{#1}}}
\newcommand{\hconvskew}[2]{\mathbin{\boxplus^{H,#2}_{#1}}}
\newcommand{\G}{\mathscr G}
\newcommand{\Et}{\mathcal E}
\newcommand{\Jd}{\mathcal J_d}
\newcommand{\Ut}{\mathcal U}
\newcommand{\ArH}{\mathcal A_{r,H}^{(t)}}
\newcommand{\Ar}{\mathcal A_r}
\newcommand{\supp}{\operatorname{supp}}

\newcommand{\Law}{\operatorname{Law}}
\newcommand{\E}{\mathbb E}
\newcommand{\id}{\operatorname{id}}
\newcommand{\1}{\mathbf 1}
\newcommand{\coeff}[2]{[z^{#1}]\,#2}

\title[Deformed convolution in Hurwitz series rings]
{Deformed Convolution, Cumulant Transforms, and Semigroup Generators in Hurwitz Series Rings}

\author{Morteza Ahmadi}
\address{Department of Pure Mathematics, Faculty of Mathematical Sciences\\
	Tarbiat Modares University, P.O.Box:14115-134, Tehran, Iran}
\email{morteza.ahmadi23@gmail.com\\ mortezy.ahmadi@modares.ac.ir}

\date{}


\begin{document}

\begin{abstract}
Let $R$ be a commutative $\Q$-algebra and let $HR$ denote its Hurwitz series ring.  We introduce a one-parameter convolution $\hconv{t}$ on $HR$ for which the specialization $t=-1$ is the intrinsic Hurwitz product.  Positive integral parameters act on finite-support truncations and reproduce finite free convolution in coefficient coordinates.  A logarithmic transform yields additive and homogeneous cumulants, an explicit inversion formula, binomial, Hermite, Laguerre, and hypergeometric families, and coefficientwise analogues of the law of large numbers and the central limit theorem.  At $t=-1$, addition of independent random variables becomes multiplication of their moment sequences in $HR$; this gives applications to classical cumulants, beta--gamma products, and self-decomposable laws.  Weighted coefficient norms turn $\hconv{t}$ into a commutative Banach algebra product.  Norm-continuous convolution semigroups then have generators conjugate to multiplication operators.  Explicit formulas are obtained for the Hermite and Laguerre semigroups, the finite free heat generator, and the L\'evy--Khintchine generator.
\end{abstract}

\maketitle

\section{Introduction}

Hurwitz series provide a coefficient algebra in which the binomial factors occurring in repeated differentiation are built directly into multiplication.  Their categorical origin and their role in differential algebra were developed by Keigher and by Keigher--Pritchard \cite{Keigher1975,Keigher1997,KeigherPritchard2000}.  The skew extension $(HR,\alpha)$ has subsequently been studied from several ring-theoretic viewpoints, including Armendariz conditions, radicals, nilpotent elements, singular ideals, and McCoy-type properties \cite{Ahmadi2019,AhmadiMoussaviNourozi2014,AhmadiMoussaviNourozi2015, NouroziMoussavi2015,NouroziMoussaviAhmadi2017,NouroziRahmatiAhmadi2021}.  We use the notation for $(HR,\alpha)$, its finite-support subring $(hR,\alpha)$, and the coefficient elements $h_n$ and $h'_r$ adopted in \cite{Ahmadi2019}.

Finite free convolution arose in the study of expected characteristic polynomials and finite-dimensional random matrix models \cite{Marcus2021,MarcusSpielmanSrivastava2022}.  Its cumulants were introduced and characterized by Arizmendi and Perales \cite{ArizmendiPerales2018}; they converge to free cumulants under natural asymptotic hypotheses.  The parameter-dependent coefficient law considered in \cite{TsujieUeda2026} suggests a natural operation on Hurwitz coefficients.  In this setting, the parameter value $t=-1$ has a particularly direct algebraic interpretation: it is exactly multiplication in the Hurwitz series ring.

The main contributions are as follows.  First, we formulate the deformation inside $(HR,\alpha)$ notation and show that its skew version is associative whenever the scalar weights are fixed by $\alpha$; at $t=-1$ it becomes the skew Hurwitz product.  Second, we construct an additive cumulant transform and give a coefficient inversion formula, which supplies a rigorous passage from cumulant convergence to coordinatewise convergence.  Third, we identify several canonical one-parameter families and record their generalized-hypergeometric representation, with the $\bm b$-parameters in the numerator and the $\bm a$-parameters in the denominator.  Finally, we place the convolution on a weighted Wiener algebra.  This permits a norm-level generator theorem and yields explicit finite free and L\'evy--Khintchine evolutions.

The paper is organized as follows.  Section~\ref{sec:ring} fixes the Hurwitz notation, defines the deformed products, and records the finite free truncation.  Section~\ref{sec:cumulants} develops cumulants, examples, hypergeometric families, and limit theorems.  Section~\ref{sec:probability} treats the specialization $t=-1$ in probability.  Section~\ref{sec:generators} studies Banach algebras and infinitesimal generators.

\section{Hurwitz series and deformed products}\label{sec:ring}

\subsection{The skew Hurwitz notation}

Throughout, $\N=\{0,1,2,\ldots\}$.  We begin with the general skew construction before specializing to commutative coefficients.

\begin{definition}\label{def:skew-Hurwitz}
Let $R$ be an associative ring with identity and let $\alpha:R\to R$ be a unital endomorphism.  The skew Hurwitz series ring $(HR,\alpha)$ consists of all maps $f:\N\to R$.  Addition is componentwise, and multiplication is defined by
\begin{equation}\label{eq:skew-product}
 (fg)(n)=\sum_{k=0}^{n}\binom{n}{k}f(k)\alpha^k\bigl(g(n-k)\bigr),
 \qquad n\in\N.
\end{equation}
For $0\ne f\in(HR,\alpha)$, put
\[
 \supp(f)=\{n\in\N:f(n)\ne0\},\qquad
 \Pi(f)=\min\supp(f),
\]
and write $\Delta(f)=\max\supp(f)$ when the maximum exists.

For $n\ge1$, let $h_n:\N\to R$ be determined by $h_n(n-1)=1$ and $h_n(m)=0$ for $m\ne n-1$.  For $r\in R$, let $h'_r(0)=r$ and $h'_r(m)=0$ for $m\ne0$.  The identity of $(HR,\alpha)$ is $h_1$.  The skew Hurwitz polynomial ring is
\[
 (hR,\alpha)=\{f\in(HR,\alpha):\Delta(f)<\infty\}.
\]
\end{definition}

Every $f\in(HR,\alpha)$ has the coordinate expansion
\begin{equation}\label{eq:coordinate-expansion}
 f=\sum_{n\ge0}h'_{f(n)}h_{n+1},
\end{equation}
where the sum is interpreted coefficientwise; it is finite for $f\in(hR,\alpha)$.  Formula \eqref{eq:skew-product} gives the useful multiplication rule
\begin{equation}\label{eq:basis-product}
 (h'_a h_{i+1})(h'_b h_{j+1})
 =\binom{i+j}{i}h'_{a\alpha^i(b)}h_{i+j+1}
 \qquad(a,b\in R;\ i,j\in\N).
\end{equation}
These conventions agree with \cite{Ahmadi2019}; related structural uses of the same notation appear in \cite{AhmadiMoussaviNourozi2014,AhmadiMoussaviNourozi2015}.

For the remainder of Sections~\ref{sec:ring}--\ref{sec:cumulants}, unless a skew endomorphism is explicitly displayed, $R$ is a commutative $\Q$-algebra, $\alpha=\id_R$, and we abbreviate $(HR,\id_R)$ and $(hR,\id_R)$ by $HR$ and $hR$.

\subsection{The parameter-dependent Hurwitz convolution}

For $a\in R$ and $n\in\N$, set
\[
 \fall{a}{0}=1,\qquad
 \fall{a}{n}=a(a-1)\cdots(a-n+1)\quad(n\ge1).
\]
Define the admissible parameter set
\[
 \Omega_R=\{t\in R:\fall{t}{n}\in R^\times\text{ for every }n\ge1\}.
\]
If $R=\R$ or $\C$, then $\Omega_R=R\setminus\Z_{\ge0}$.

\begin{definition}\label{def:t-product}
For $t\in\Omega_R$ and $f,g\in HR$, define
\begin{equation}\label{eq:t-product}
 (f\hconv{t}g)(n)
 =\sum_{i+j=n}\frac{\fall{t}{n}}{\fall{t}{i}\fall{t}{j}}f(i)g(j),
 \qquad n\in\N.
\end{equation}
\end{definition}

Let $R[[z]]$ denote the coefficient algebra of series in an indeterminate $z$, with the Cauchy product.  Introduce
\begin{equation}\label{eq:Phi}
 \Phi_t^H:HR\longrightarrow R[[z]],\qquad
 \Phi_t^H(f)(z)=\sum_{n\ge0}\frac{f(n)}{\fall{t}{n}}z^n.
\end{equation}

\begin{proposition}\label{prop:algebra-isomorphism}
For every $t\in\Omega_R$, the operation $\hconv{t}$ is bilinear, commutative, and associative, with identity $h_1$.  Moreover, $\Phi_t^H$ is an algebra isomorphism and
\begin{equation}\label{eq:Phi-multiplicative}
 \Phi_t^H(f\hconv{t}g)=\Phi_t^H(f)\Phi_t^H(g).
\end{equation}
At $t=-1$, the operation $\hconv{-1}$ is the ordinary Hurwitz product on $HR$.
\end{proposition}

\begin{proof}
The coefficient of $z^n$ in the product on the right of \eqref{eq:Phi-multiplicative} is
\[
 \sum_{i+j=n}\frac{f(i)}{\fall{t}{i}}\frac{g(j)}{\fall{t}{j}}
 =\frac{(f\hconv{t}g)(n)}{\fall{t}{n}}.
\]
Thus \eqref{eq:Phi-multiplicative} holds.  Since $\Phi_t^H$ is a coefficientwise bijection, the algebraic properties follow from the Cauchy product.  Finally,
\[
 \frac{\fall{-1}{n}}{\fall{-1}{i}\fall{-1}{j}}
 =\frac{(-1)^n n!}{(-1)^i i!\,(-1)^j j!}
 =\binom{n}{i}
 \quad(i+j=n),
\]
which is the coefficient in the Hurwitz multiplication law.
\end{proof}

The same weights can be combined with an endomorphism.

\begin{proposition}
	\label{prop:skew-deformation}
Let $R$ be an algebra over a commutative ring $S$, let $\alpha:R\to R$ be an $S$-linear endomorphism, and choose $t\in S$ so that all $\fall{t}{n}$ are units in $S$.  For $f,g\in(HR,\alpha)$, define
\begin{equation}\label{eq:skew-t-product}
 (f\hconvskew{t}{\alpha}g)(n)
 =\sum_{i+j=n}\frac{\fall{t}{n}}{\fall{t}{i}\fall{t}{j}}
 f(i)\alpha^i(g(j)).
\end{equation}
Then $\hconvskew{t}{\alpha}$ is associative and has identity $h_1$.  Its specialization at $t=-1$ is the multiplication of $(HR,\alpha)$ in Definition~\ref{def:skew-Hurwitz}.
\end{proposition}

\begin{proof}
Write $c_{i,j}=\fall{t}{i+j}/(\fall{t}{i}\fall{t}{j})$.  The cocycle identity
\[
 c_{i,j}c_{i+j,k}=c_{j,k}c_{i,j+k}
 =\frac{\fall{t}{i+j+k}}{\fall{t}{i}\fall{t}{j}\fall{t}{k}}
\]
and $S$-linearity of $\alpha$ show that the coefficients of
$(f\hconvskew{t}{\alpha}g)\hconvskew{t}{\alpha}u$ and
$f\hconvskew{t}{\alpha}(g\hconvskew{t}{\alpha}u)$ coincide.  The identity assertion is immediate.  The final statement follows from the computation in the proof of Proposition~\ref{prop:algebra-isomorphism}.
\end{proof}

\subsection{Positive integral parameters and finite free convolution}

For $d\ge1$, put
\[
 h_{\le d}R=\{f\in hR:f(n)=0\text{ for }n>d\}.
\]
If $f,g\in h_{\le d}R$, define
\begin{equation}\label{eq:d-product}
 (f\hconv{d}g)(n)
 =\sum_{i+j=n}\frac{\fall{d}{n}}{\fall{d}{i}\fall{d}{j}}f(i)g(j),
 \qquad 0\le n\le d,
\end{equation}
with all coefficients above $d$ set equal to zero.  The map
\[
 \Phi_d^{H,[d]}(f)=\sum_{n=0}^{d}\frac{f(n)}{\fall{d}{n}}z^n
\]
identifies $(h_{\le d}R,\hconv{d})$ with the quotient algebra $R[z]/(z^{d+1})$.

For $p(x)=\sum_{k=0}^{d}a_kx^{d-k}\in R[x]_d$, define
\begin{equation}\label{eq:Jd}
 \Jd(p)=\sum_{k=0}^{d}h'_{a_k}h_{k+1}\in h_{\le d}R,
 \qquad\text{equivalently,}\qquad \Jd(p)(k)=a_k.
\end{equation}
The finite free convolution is
\begin{equation}\label{eq:finite-free}
 (p\boxplus_d q)(x)
 =\sum_{k=0}^{d}\sum_{i+j=k}
 \frac{\fall{d}{k}}{\fall{d}{i}\fall{d}{j}}a_i b_jx^{d-k}.
\end{equation}
Consequently,
\begin{equation}\label{eq:J-intertwines}
 \Jd(p\boxplus_d q)=\Jd(p)\hconv{d}\Jd(q).
\end{equation}
Thus finite free convolution is the positive-integral truncation of the Hurwitz coefficient law.  For background on finite free convolution and its random-matrix interpretation, see \cite{Marcus2021,MarcusSpielmanSrivastava2022}.

\section{Cumulants and canonical Hurwitz families}\label{sec:cumulants}

\subsection{The cumulant transform}

Let
\[
 \Ut(HR)=\{f\in HR:f(0)=1\},\qquad
 \Ut_d(HR)=\Ut(HR)\cap h_{\le d}R.
\]
For coefficient calculations, write
\begin{equation}\label{eq:G-transform}
 \G f(z)=\sum_{n\ge0}f(n)z^n.
\end{equation}
For $f\in\Ut(HR)$, define its power-sum coordinates $p_n^H(f)$ by
\begin{equation}\label{eq:power-sums}
 P_H[f](z)=\sum_{n\ge1}p_n^H(f)z^n
 =-z\frac{d}{dz}\log\bigl(\G f(z)\bigr).
\end{equation}
The logarithm is the coefficientwise logarithm of a series with constant term one.

For $t\in\Omega_R$, define
\begin{equation}\label{eq:E-transform}
 (\Et_t f)(0)=1,\qquad
 (\Et_t f)(n)=\frac{t^n}{\fall{t}{n}}f(n)\quad(n\ge1).
\end{equation}

\begin{definition}\label{def:cumulants}
For $f\in\Ut(HR)$ and $t\in\Omega_R$, the $t$-deformed Hurwitz cumulant transform is
\begin{equation}\label{eq:K-transform}
 K_t^H[f](z)=\sum_{n\ge1}\kappa_n^t(f)z^n
 =-\frac{z}{t}\frac{d}{dz}\log\bigl(\G(\Et_t f)(z)\bigr).
\end{equation}
For $f\in\Ut_d(HR)$ and $d\in\Z_{\ge1}$, put
\begin{equation}\label{eq:K-d}
 K_d^H[f](z)
 =-\frac{z}{d}\frac{d}{dz}
 \log\left(1+\sum_{n=1}^{d}\frac{d^n}{\fall{d}{n}}f(n)z^n\right).
\end{equation}
\end{definition}

The next identity is both an inversion formula and the key tool for the limit theorems.

\begin{proposition}
	\label{prop:inversion}
For $t\in\Omega_R$ and $f\in\Ut(HR)$,
\begin{equation}\label{eq:inversion-series}
 \G(\Et_t f)(z)
 =\exp\left(-t\sum_{n\ge1}\frac{\kappa_n^t(f)}{n}z^n\right).
\end{equation}
In particular,
\begin{equation}\label{eq:coefficient-inversion}
 f(n)=\frac{\fall{t}{n}}{t^n}
 \coeff{n}{\exp\left(-t\sum_{j=1}^{n}\frac{\kappa_j^t(f)}{j}z^j\right)}.
\end{equation}
Hence $K_t^H$ is injective, and each coefficient $f(n)$ is a polynomial in the first $n$ cumulants.
\end{proposition}

\begin{proof}
Divide \eqref{eq:K-transform} by $z$ and integrate coefficientwise from $0$ to $z$.  The integration constant is zero because $\G(\Et_t f)(0)=1$.  This gives \eqref{eq:inversion-series}; coefficient extraction yields \eqref{eq:coefficient-inversion}.
\end{proof}

For $r\in R$, define the dilation $D_r f$ by $(D_r f)(n)=r^nf(n)$.

\begin{theorem}[Cumulant characterization]\label{thm:cumulants}
For $t\in\Omega_R$, the maps $\kappa_n^t:\Ut(HR)\to R$ have the following properties.
\begin{enumerate}[label=\textup{(\roman*)},leftmargin=2.2em]
\item For every $n\ge1$, $\kappa_n^t(f)$ is a polynomial in
$p_1^H(f),\ldots,p_n^H(f)$ whose $p_n^H$-term is
\begin{equation}\label{eq:leading-term}
 \frac{t^{n-1}}{\fall{t}{n}}p_n^H(f).
\end{equation}
\item $\kappa_n^t(D_r f)=r^n\kappa_n^t(f)$.
\item $\kappa_n^t(f\hconv{t}g)=\kappa_n^t(f)+\kappa_n^t(g)$.
\end{enumerate}
Among homogeneous polynomial functionals satisfying additivity, the leading term \eqref{eq:leading-term} determines $\kappa_n^t$ uniquely.  The same assertions hold for $t=d$ on $\Ut_d(HR)$.
\end{theorem}

\begin{proof}
The identities
\[
 P_H[f](z)\G f(z)=-z\frac{d}{dz}\G f(z)
\]
and
\[
 tK_t^H[f](z)\G(\Et_t f)(z)
 =-z\frac{d}{dz}\G(\Et_t f)(z)
\]
give triangular recurrences after coefficients are compared.  Induction on $n$ shows that the coefficient of $p_n^H(f)$ in $\kappa_n^t(f)$ is $t^{n-1}/\fall{t}{n}$ and that all remaining terms involve only $p_1^H(f),\ldots,p_{n-1}^H(f)$.

Since $\G(D_r f)(z)=\G f(rz)$ and $\Et_tD_r=D_r\Et_t$, equation \eqref{eq:K-transform} gives $K_t^H[D_r f](z)=K_t^H[f](rz)$, proving homogeneity.  Furthermore,
\begin{align*}
 \G\bigl(\Et_t(f\hconv{t}g)\bigr)(z)
 &=\sum_{n\ge0}\frac{t^n}{\fall{t}{n}}(f\hconv{t}g)(n)z^n\\
 &=\G(\Et_t f)(z)\G(\Et_t g)(z).
\end{align*}
Taking the logarithm in \eqref{eq:K-transform} proves additivity.

For uniqueness, the coordinate Hopf algebra of principal units under $\hconv{t}$ is obtained from the ordinary multiplicative coordinate Hopf algebra by the invertible diagonal change $f\mapsto\Et_t f$.  Over a $\Q$-algebra, its primitive homogeneous component of degree $n$ is one-dimensional; equivalently, it is generated by the $n$th power-sum primitive.  The coefficient in \eqref{eq:leading-term} therefore fixes the functional.  This is the standard cumulant uniqueness argument; see \cite{Macdonald1995,Lehner2002,ArizmendiPerales2018}.
\end{proof}

If $p\in R[x]_{\mathrm{monic},d}$, then Definition~\ref{def:cumulants} and \eqref{eq:Jd} imply
\begin{equation}\label{eq:finite-free-cumulants}
 K_d^H[\Jd(p)](z)
 \equiv\sum_{n=1}^{d}\kappa_n^{(d)}(p)z^n\pmod{z^{d+1}},
\end{equation}
where $\kappa_n^{(d)}$ are the finite free cumulants of \cite{ArizmendiPerales2018}.

\subsection{Binomial, Hermite, and Laguerre series}

The infinite-series formulas below use $t\in\R\setminus\Z_{\ge0}$, so that $t\in\Omega_{\R}$.  Their positive-integral counterparts are interpreted in $h_{\le d}R$ by retaining the coefficients with $0\le n\le d$.

\begin{example}[Binomial family]\label{ex:binomial}
For $t\in\R\setminus\Z_{\ge0}$ and $\lambda\in\R$, define
\begin{equation}\label{eq:binomial}
 \mathsf B_{\lambda}^{(t)}(n)
 =(-1)^n\frac{\fall{t}{n}}{n!}\lambda^n.
\end{equation}
Then $\G\mathsf B_{\lambda}^{(t)}(z)=(1-\lambda z)^t$ and
\[
 K_t^H[\mathsf B_{\lambda}^{(t)}](z)=\lambda z.
\]
Consequently,
\begin{equation}\label{eq:binomial-semigroup}
 \mathsf B_{\lambda_1}^{(t)}\hconv{t}\mathsf B_{\lambda_2}^{(t)}
 =\mathsf B_{\lambda_1+\lambda_2}^{(t)}.
\end{equation}
For $d\ge1$, $\mathsf B_{\lambda}^{(d)}=\Jd((x-\lambda)^d)$.
\end{example}

\begin{example}[Hermite family]\label{ex:Hermite}
For $t\in\R\setminus\Z_{\ge0}$, define
\begin{equation}\label{eq:Hermite}
 \mathsf H^{(t)}(2k)
 =(-1)^k\frac{\fall{t}{2k}}{t^k(2k)!!},
 \qquad
 \mathsf H^{(t)}(2k+1)=0.
\end{equation}
Then
\[
 K_t^H[\mathsf H^{(t)}](z)=z^2.
\]
For $s\ge0$, set $\mathsf H_s^{(t)}=D_{\sqrt{s}}\mathsf H^{(t)}$ and $\mathsf H_0^{(t)}=h_1$.  The cumulant transform gives
\begin{equation}\label{eq:Hermite-semigroup}
 \mathsf H_{s_1}^{(t)}\hconv{t}\mathsf H_{s_2}^{(t)}
 =\mathsf H_{s_1+s_2}^{(t)}.
\end{equation}
For $d\ge1$, $\mathsf H^{(d)}=\Jd(H_d)$, where
\[
 H_d(x)=\sum_{k=0}^{\lfloor d/2\rfloor}
 (-1)^k\frac{\fall{d}{2k}}{d^k(2k)!!}x^{d-2k}.
\]
\end{example}

\begin{example}[Laguerre family]\label{ex:Laguerre}
For $t\in\R\setminus\Z_{\ge0}$ and $\lambda\ge0$, define
\begin{equation}\label{eq:Laguerre}
 \mathsf L_{\lambda}^{(t)}(n)
 =(-1)^n\frac{\fall{\lambda t}{n}\fall{t}{n}}{t^n n!}.
\end{equation}
Then
\[
 K_t^H[\mathsf L_{\lambda}^{(t)}](z)
 =\frac{\lambda z}{1-z},
\]
so every cumulant equals $\lambda$ and
\begin{equation}\label{eq:Laguerre-semigroup}
 \mathsf L_{\lambda_1}^{(t)}\hconv{t}\mathsf L_{\lambda_2}^{(t)}
 =\mathsf L_{\lambda_1+\lambda_2}^{(t)}.
\end{equation}
For $d\ge1$, $\mathsf L_{\lambda}^{(d)}=\Jd(L_{d,\lambda})$, with
\[
 L_{d,\lambda}(x)=\sum_{k=0}^{d}
 (-1)^k\frac{\fall{\lambda d}{k}\fall{d}{k}}{d^k k!}x^{d-k}.
\]
\end{example}

\begin{proof}[Verification of Examples~\ref{ex:binomial}--\ref{ex:Laguerre}]
The transformed coefficient series in \eqref{eq:K-transform} are, respectively,
\[
 \G(\Et_t\mathsf B_{\lambda}^{(t)})(z)=e^{-t\lambda z},
 \qquad
 \G(\Et_t\mathsf H^{(t)})(z)=e^{-tz^2/2},
\]
and
\[
 \G(\Et_t\mathsf L_{\lambda}^{(t)})(z)=(1-z)^{\lambda t}.
\]
Applying $-(z/t)d/dz$ gives the three cumulant transforms.  The semigroup identities follow from Theorem~\ref{thm:cumulants} and Proposition~\ref{prop:inversion}.  The positive-integral identifications follow by comparing coefficients under $\Jd$.
\end{proof}

\subsection{Hypergeometric Hurwitz series}

For the infinite-series identities in this subsection, let $R=\C$ and $t\in\C\setminus\Z_{\ge0}$.  For vectors $\bm a=(a_1,\ldots,a_i)$ and $\bm b=(b_1,\ldots,b_j)$, write
\[
 \fall{\bm a}{n}=\prod_{r=1}^{i}\fall{a_r}{n},
 \qquad
 \rise{\bm a}{n}=\prod_{r=1}^{i}\rise{a_r}{n}.
\]
Assume that no factor in $\fall{t\bm a}{n}$ vanishes.  Define
\begin{equation}\label{eq:hyper-H}
 \mathfrak H_t\!\begin{bmatrix}\bm b\\\bm a\end{bmatrix}(n)
 =(-1)^n\frac{\fall{t}{n}\fall{t\bm b}{n}}
 {n!\fall{t\bm a}{n}}.
\end{equation}
With the standard generalized-hypergeometric notation,
\begin{equation}\label{eq:hyper-correct}
 \G\mathfrak H_t\!\begin{bmatrix}\bm b\\\bm a\end{bmatrix}(z)
 ={}_{j+1}F_i\!\left(
 \begin{matrix}-t,-t\bm b\\-t\bm a\end{matrix};
 (-1)^{j-i}z\right).
\end{equation}
The placement of $-t\bm b$ in the numerator and $-t\bm a$ in the denominator follows directly from \eqref{eq:hyper-H}.

For $d\ge1$, the corresponding polynomial is
\begin{equation}\label{eq:hyper-poly}
 \mathcal H_d\!\begin{bmatrix}\bm b\\\bm a\end{bmatrix}(x)
 =\sum_{k=0}^{d}(-1)^k
 \frac{\fall{d}{k}\fall{d\bm b}{k}}
 {k!\fall{d\bm a}{k}}x^{d-k},
\end{equation}
whenever the denominators are nonzero, and
\[
 \Jd\left(\mathcal H_d\!\begin{bmatrix}\bm b\\\bm a\end{bmatrix}\right)
 =\mathfrak H_d\!\begin{bmatrix}\bm b\\\bm a\end{bmatrix}.
\]
Up to coefficient dilation, these families include the binomial and Laguerre cases and connect with the finite free study of hypergeometric polynomials \cite{KoekoekLeskySwarttouw2010,MartinezMoralesPerales2024,CampbellMoralesPerales2025}.

\begin{proposition}
	\label{prop:hyper-closure}
For $\ell=1,2,3$, let $\bm a_\ell$ and $\bm b_\ell$ have lengths $i_\ell$ and $j_\ell$, assume that every denominator parameter is admissible, and put $s_\ell=(-1)^{i_\ell+j_\ell+1}$.  Then
\begin{equation}\label{eq:hyper-identity}
 {}_{j_1}F_{i_1}\!\left(\begin{matrix}-t\bm b_1\\-t\bm a_1\end{matrix};z\right)
 {}_{j_2}F_{i_2}\!\left(\begin{matrix}-t\bm b_2\\-t\bm a_2\end{matrix};z\right)
 ={}_{j_3}F_{i_3}\!\left(\begin{matrix}-t\bm b_3\\-t\bm a_3\end{matrix};z\right)
\end{equation}
if and only if
\begin{equation}\label{eq:hyper-convolution}
 D_{s_1}\mathfrak H_t\!\begin{bmatrix}\bm b_1\\\bm a_1\end{bmatrix}
 \hconv{t}
 D_{s_2}\mathfrak H_t\!\begin{bmatrix}\bm b_2\\\bm a_2\end{bmatrix}
 =D_{s_3}\mathfrak H_t\!\begin{bmatrix}\bm b_3\\\bm a_3\end{bmatrix}.
\end{equation}
\end{proposition}

\begin{proof}
From \eqref{eq:Phi}, \eqref{eq:hyper-H}, and the conversion
$\fall{-c}{n}=(-1)^n\rise{c}{n}$, one obtains
\[
 \Phi_t^H\left(D_{(-1)^{i+j+1}}
 \mathfrak H_t\!\begin{bmatrix}\bm b\\\bm a\end{bmatrix}\right)(z)
 ={}_{j}F_i\!\left(\begin{matrix}-t\bm b\\-t\bm a\end{matrix};z\right).
\]
Apply the algebra isomorphism in Proposition~\ref{prop:algebra-isomorphism}.
\end{proof}

\subsection{Coefficientwise limit theorems}

Equip $HR\cong\prod_{n\ge0}R$ with the product topology.  Thus $f_m\to f$ means $f_m(n)\to f(n)$ for each fixed $n$.

\begin{lemma}\label{lem:cumulant-convergence}
Let $R=\R$ or $\C$ and $t\in\Omega_R$.  For a sequence $f_m\in\Ut(HR)$ and $f\in\Ut(HR)$, the following are equivalent:
\begin{enumerate}[label=\textup{(\roman*)},leftmargin=2.2em]
\item $f_m(n)\to f(n)$ for every $n$;
\item $\kappa_n^t(f_m)\to\kappa_n^t(f)$ for every $n$.
\end{enumerate}
\end{lemma}

\begin{proof}
Each cumulant is a polynomial in the first finitely many coefficients by Definition~\ref{def:cumulants}, so (i) implies (ii).  Conversely, formula \eqref{eq:coefficient-inversion} expresses $f_m(n)$ as a fixed polynomial in $\kappa_1^t(f_m),\ldots,\kappa_n^t(f_m)$, proving (ii)$\Rightarrow$(i).
\end{proof}

For $m\ge1$, write $f^{\hconv{t}m}$ for the $m$-fold convolution power.

\begin{theorem}[Law of large numbers and central limit theorem]\label{thm:LLN-CLT}
Let $t\in\R\setminus\Z_{\ge0}$ and $f\in\Ut(H\R)$.
\begin{enumerate}[label=\textup{(\roman*)},leftmargin=2.2em]
\item If $\kappa_1^t(f)=\lambda$, then
\[
 D_{1/m}\bigl(f^{\hconv{t}m}\bigr)
 \longrightarrow \mathsf B_{\lambda}^{(t)}.
\]
\item If $\kappa_1^t(f)=0$ and $\kappa_2^t(f)=1$, then
\[
 D_{1/\sqrt m}\bigl(f^{\hconv{t}m}\bigr)
 \longrightarrow \mathsf H^{(t)}.
\]
\end{enumerate}
Both limits are coefficientwise as $m\to\infty$.
\end{theorem}

\begin{proof}
Additivity and homogeneity give
\[
 \kappa_n^t\left(D_{1/m}(f^{\hconv{t}m})\right)
 =m^{1-n}\kappa_n^t(f).
\]
The right side tends to $\lambda$ for $n=1$ and to zero for $n\ge2$, matching the cumulants of $\mathsf B_{\lambda}^{(t)}$.  Lemma~\ref{lem:cumulant-convergence} proves (i).  Similarly,
\[
 \kappa_n^t\left(D_{1/\sqrt m}(f^{\hconv{t}m})\right)
 =m^{1-n/2}\kappa_n^t(f),
\]
which converges to the cumulant sequence with only the second component equal to one.  This is the cumulant sequence of $\mathsf H^{(t)}$, proving (ii).
\end{proof}

\section{\texorpdfstring{The specialization $t=-1$ in probability}{The specialization t=-1 in probability}}\label{sec:probability}

In this section, coefficients are real.  If a random variable $X$ has moments of every order, define its Hurwitz moment series by
\begin{equation}\label{eq:moment-series}
 \mu_X(n)=\E[X^n],\qquad n\in\N.
\end{equation}

\subsection{Independent sums and classical cumulants}

Let
\[
 \Psi_X(z)=\sum_{n\ge0}\frac{\E[X^n]}{n!}z^n
\]
be the exponential moment series.  Define the classical cumulants $c_n(X)$ by
\begin{equation}\label{eq:classical-cumulants}
 \log\Psi_X(z)=\sum_{n\ge1}\frac{c_n(X)}{n!}z^n.
\end{equation}

\begin{theorem}\label{thm:classical}
Let $X$ and $Y$ have moments of every order.
\begin{enumerate}[label=\textup{(\roman*)},leftmargin=2.2em]
\item If $X$ and $Y$ are independent, then
\begin{equation}\label{eq:independent-Hurwitz}
 \mu_{X+Y}=\mu_X\hconv{-1}\mu_Y=\mu_X\mu_Y
 \quad\text{in }H\R.
\end{equation}
\item The Hurwitz cumulants and the classical cumulants satisfy
\begin{equation}\label{eq:cumulant-relation}
 K_{-1}^H[\mu_X](z)
 =z\frac{d}{dz}\log\Psi_X(z)
 =\sum_{n\ge1}\frac{c_n(X)}{(n-1)!}z^n.
\end{equation}
In particular,
\begin{equation}\label{eq:cumulant-scaling}
 \kappa_n^{-1}(\mu_X)=\frac{c_n(X)}{(n-1)!}.
\end{equation}
\end{enumerate}
\end{theorem}

\begin{proof}
The coefficient of index $n$ in the right side of \eqref{eq:independent-Hurwitz} is
\[
 \sum_{k=0}^{n}\binom{n}{k}\E[X^k]\E[Y^{n-k}]
 =\E[(X+Y)^n],
\]
where independence is used in the second equality.  Since
$\fall{-1}{n}=(-1)^n n!$, equation \eqref{eq:E-transform} gives
$\G(\Et_{-1}\mu_X)=\Psi_X$.  Substitution in \eqref{eq:K-transform} proves \eqref{eq:cumulant-relation}, and coefficient comparison gives \eqref{eq:cumulant-scaling}.
\end{proof}

If $C_X(u)=\log\E[e^{iuX}]$ is defined near $u=0$, then \eqref{eq:cumulant-relation} also gives
\begin{equation}\label{eq:characteristic-K}
 K_{-1}^H[\mu_X](iu)=u\frac{d}{du}C_X(u).
\end{equation}

\subsection{Products of beta and gamma variables}

For positive random variables, let $\boxtimes_{\mathrm{cl}}$ denote multiplicative convolution, namely the law of a product of independent variables.

\begin{proposition}\label{prop:beta-gamma}
Let $j\ge i\ge1$ and assume $j-i$ is even.  Let
$\bm a=(a_1,\ldots,a_i)$ and $\bm b=(b_1,\ldots,b_j)$ satisfy
$0<b_r<a_r$ for $1\le r\le i$ and $b_r>0$ for $i<r\le j$.
Then
\[
 \mathfrak H_{-1}\!\begin{bmatrix}\bm b\\\bm a\end{bmatrix}=\mu_X,
\]
where
\begin{multline}\label{eq:beta-gamma-law}
 \Law(X)=
 \operatorname{Beta}(b_1,a_1-b_1)\boxtimes_{\mathrm{cl}}\cdots
 \boxtimes_{\mathrm{cl}}\operatorname{Beta}(b_i,a_i-b_i)\\
 \boxtimes_{\mathrm{cl}}\operatorname{Gamma}(b_{i+1},1)
 \boxtimes_{\mathrm{cl}}\cdots
 \boxtimes_{\mathrm{cl}}\operatorname{Gamma}(b_j,1).
\end{multline}
\end{proposition}

\begin{proof}
The $n$th moment of the product in \eqref{eq:beta-gamma-law} is
\[
 \frac{\rise{\bm b}{n}}{\rise{\bm a}{n}}.
\]
Using $\fall{-c}{n}=(-1)^n\rise{c}{n}$ in \eqref{eq:hyper-H}, the parity condition gives
\[
 \mathfrak H_{-1}\!\begin{bmatrix}\bm b\\\bm a\end{bmatrix}(n)
 =\frac{\rise{\bm b}{n}}{\rise{\bm a}{n}}.
\]
\end{proof}

For $\lambda>0$, substituting $t=-1$ in \eqref{eq:Laguerre} and using the gamma moment formula $\E[X^n]=\rise{\lambda}{n}$ gives
\begin{equation}\label{eq:gamma-Laguerre}
 \mathsf L_{\lambda}^{(-1)}=\mu_X,
 \qquad X\sim\operatorname{Gamma}(\lambda,1).
\end{equation}

\subsection{Self-decomposable laws}

A random variable $X$ is self-decomposable if, for each $c\in(0,1)$, one can write
$X\stackrel{d}{=}cX+X_c$ with $X_c$ independent of $X$.  Such a law has a background driving L\'evy process $Z$ satisfying the random integral representation
\[
 X\stackrel{d}{=}\int_0^\infty e^{-s}\,dZ(s),
\]
and its characteristic exponent obeys
\begin{equation}\label{eq:BDLP-classical}
 C_{Z(s)}(u)=su\frac{d}{du}C_X(u).
\end{equation}
See \cite{JurekVervaat1983,Jurek2001}.

\begin{proposition}\label{prop:BDLP}
Let $X$ be self-decomposable, have moments of every order, and have background driving process $Z$.  Then
\begin{equation}\label{eq:BDLP-Hurwitz}
 C_{Z(s)}(u)=sK_{-1}^H[\mu_X](iu),\qquad s\ge0.
\end{equation}
\end{proposition}

\begin{proof}
Combine \eqref{eq:BDLP-classical} with \eqref{eq:characteristic-K}.
\end{proof}

\begin{example}\label{ex:BDLP-examples}
\begin{enumerate}[label=\textup{(\roman*)},leftmargin=2.2em]
\item If $X\sim N(0,1)$, then $\mu_X=\mathsf H^{(-1)}$.  Since
$K_{-1}^H[\mathsf H^{(-1)}](iu)=-u^2$, the background driving process is a Brownian motion with variance rate $2$, equivalently $Z(s)=\sqrt2\,B(s)$ for a standard Brownian motion $B$.
\item If $X\sim\operatorname{Gamma}(\lambda,1)$, then
$\mu_X=\mathsf L_{\lambda}^{(-1)}$.  Its background driving process is a compound Poisson process with rate $\lambda$ and unit-exponential jumps \cite{Jurek2001}.
\end{enumerate}
\end{example}

\section{Banach algebras and infinitesimal generators}\label{sec:generators}

\subsection{A weighted Hurwitz algebra}

Fix $t\in\R\setminus\Z_{\ge0}$ and $r>0$.  Define
\begin{equation}\label{eq:ArH}
 \ArH
 =\left\{f\in H\R:
 \|f\|_{r,t}:=\sum_{n\ge0}
 \left|\frac{f(n)}{\fall{t}{n}}\right|r^n<\infty\right\}
\end{equation}
and let
\begin{equation}\label{eq:Ar}
 \Ar=\left\{F(z)=\sum_{n\ge0}a_nz^n:
 \|F\|_r:=\sum_{n\ge0}|a_n|r^n<\infty\right\}.
\end{equation}
The latter is the weighted Wiener algebra under the Cauchy product.

\begin{proposition}\label{prop:Banach}
The pair $(\ArH,\hconv{t})$ is a unital commutative Banach algebra, and
\begin{equation}\label{eq:Banach-Phi}
 \Phi_t^H:(\ArH,\hconv{t})\longrightarrow(\Ar,\cdot)
\end{equation}
is an isometric algebra isomorphism.
\end{proposition}

\begin{proof}
Equation \eqref{eq:Phi-multiplicative} gives multiplicativity, while
$\|\Phi_t^H(f)\|_r=\|f\|_{r,t}$.  Completeness and the submultiplicative norm follow from the corresponding facts for $\Ar$.
\end{proof}

\subsection{Norm-continuous Hurwitz convolution semigroups}

\begin{definition}\label{def:semigroup}
A family $Q=\{q_s\}_{s\ge0}\subset\ArH$ is a norm-continuous $\hconv{t}$-semigroup if
\[
 q_0=h_1,\qquad q_{s+u}=q_s\hconv{t}q_u,
 \qquad\text{and}\qquad
 \|q_s-h_1\|_{r,t}\longrightarrow0\quad(s\downarrow0).
\]
Its convolution operators are
\[
 T_s f=f\hconv{t}q_s.
\]
\end{definition}

\begin{theorem}[Generator representation]\label{thm:generator}
Let $Q$ be a norm-continuous $\hconv{t}$-semigroup in $\ArH$.  Then the limit
\begin{equation}\label{eq:eta}
 \eta_Q=\lim_{s\downarrow0}
 \frac{\Phi_t^H(q_s)-1}{s}
\end{equation}
exists in $\Ar$, and
\begin{equation}\label{eq:semigroup-exponential}
 \Phi_t^H(q_s)=\exp(s\eta_Q),\qquad s\ge0.
\end{equation}
The infinitesimal generator is the bounded operator
\begin{equation}\label{eq:generator-conjugacy}
 L_{r,H}^{(t)}[Q]
 =(\Phi_t^H)^{-1}\circ M_{\eta_Q}\circ\Phi_t^H,
\end{equation}
where $M_{\eta_Q}F=\eta_QF$.  Equivalently,
\begin{equation}\label{eq:generator-limit}
 L_{r,H}^{(t)}[Q]f
 =\lim_{s\downarrow0}\frac{f\hconv{t}q_s-f}{s}
 \quad\text{for every }f\in\ArH.
\end{equation}
\end{theorem}

\begin{proof}
Multiplication by $q_s$ has operator norm $\|q_s\|_{r,t}$, and
\[
 \|T_s-I\|=\|q_s-h_1\|_{r,t};
\]
the reverse inequality follows by applying $T_s-I$ to $h_1$.  Hence $\{T_s\}$ is a uniformly continuous operator semigroup.  The standard semigroup theorem gives a bounded generator and an operator-norm exponential representation \cite{EngelNagel2000}.  Conjugating by the isometry $\Phi_t^H$ turns $T_s$ into multiplication by $\Phi_t^H(q_s)$.  The operator-norm derivative is therefore multiplication by the $\Ar$-limit in \eqref{eq:eta}, which proves \eqref{eq:semigroup-exponential}--\eqref{eq:generator-limit}.
\end{proof}

\begin{example}[Hermite generator]\label{ex:Hermite-generator}
Let $q_s=\mathsf H_s^{(t)}$.  For every $r>0$,
\[
 \Phi_t^H(q_s)(z)=\exp\left(-\frac{s}{2t}z^2\right),
 \qquad
 \eta_Q(z)=-\frac{z^2}{2t}.
\]
If $f(n)=a_n$, then
\begin{equation}\label{eq:Hermite-generator}
 \G\bigl(L_{r,H}^{(t)}[Q]f\bigr)(z)
 =-\frac{1}{2t}\sum_{n\ge0}
 \frac{\fall{t}{n+2}}{\fall{t}{n}}a_nz^{n+2}
 =-\frac{1}{2t}\sum_{n\ge0}(t-n)(t-n-1)a_nz^{n+2}.
\end{equation}
\end{example}

\begin{example}[Laguerre generator]\label{ex:Laguerre-generator}
Let $q_s=\mathsf L_s^{(t)}$.  If $0<r<|t|$, then
\[
 \Phi_t^H(q_s)(z)=\left(1-\frac{z}{t}\right)^{st},
 \qquad
 \eta_Q(z)=t\log\left(1-\frac{z}{t}\right)
 =-\sum_{\ell\ge1}\frac{z^\ell}{\ell t^{\ell-1}}.
\]
Therefore
\begin{equation}\label{eq:Laguerre-generator}
 \G\bigl(L_{r,H}^{(t)}[Q]f\bigr)(z)
 =-\sum_{k\ge1}\sum_{\ell=1}^{k}
 \frac{\fall{t-k+\ell}{\ell}}{\ell t^{\ell-1}}
 f(k-\ell)z^k.
\end{equation}
\end{example}

\subsection{Finite free generators}

In this subsection, let $R=\R$ or $\C$.  Let $q=\{q_s\}_{s\ge0}\subset R[x]_d$ satisfy
\[
 q_0(x)=x^d,\qquad q_{s+u}=q_s\boxplus_d q_u,
 \qquad q_s\longrightarrow x^d\quad(s\downarrow0)
\]
coefficientwise.  Extend the dilation notation to $R[x]_d$ by
\[
 D_c\left(\sum_{k=0}^{d}a_kx^{d-k}\right)
 =\sum_{k=0}^{d}c^ka_kx^{d-k};
\]
then $\Jd(D_cp)=D_c\Jd(p)$.  The operators $S_s p=p\boxplus_d q_s$ form a continuous semigroup on the finite-dimensional space $R[x]_d$.  Its generator is therefore defined for every $p\in R[x]_d$ by
\[
 L[q]p=\lim_{s\downarrow0}\frac{p\boxplus_dq_s-p}{s},
\]
where the limit is coefficientwise.

\begin{theorem}\label{thm:finite-generator}
For every $p\in R[x]_d$,
\begin{equation}\label{eq:finite-generator-intertwining}
 \Jd(L[q]p)
 =L_H^{(d)}[\Jd(q)]\,\Jd(p),
\end{equation}
where the right-hand side is the generator associated with $\hconv{d}$ on $h_{\le d}R$.
\end{theorem}

\begin{proof}
Apply $\Jd$ to the difference quotient and use \eqref{eq:J-intertwines}; then pass to the coefficientwise limit.
\end{proof}

\begin{example}[Finite free heat generator]\label{ex:finite-heat}
For $q_s=D_{\sqrt s}H_d$ and
$p(x)=\sum_{k=0}^{d}a_kx^{d-k}$,
\begin{equation}\label{eq:finite-heat-H}
 \G\bigl(L_H^{(d)}[\Jd(q)]\Jd(p)\bigr)(z)
 =-\frac{1}{2d}\sum_{k=0}^{d-2}(d-k)(d-k-1)a_kz^{k+2}.
\end{equation}
Equivalently,
\begin{equation}\label{eq:finite-heat-p}
 L[q]p=-\frac{1}{2d}\frac{d^2p}{dx^2}.
\end{equation}
\end{example}

\subsection{L\'evy--Khintchine generators}

Let $X=\{X(s)\}_{s\ge0}$ be a real L\'evy process with characteristic triplet $(\gamma,a,\nu)$, using the truncation $x\1_{\{|x|\le1\}}$.  Suppose that, for some $\rho>0$,
\begin{equation}\label{eq:exp-moment}
 \int_{|x|>1}e^{\rho|x|}\,\nu(dx)<\infty.
\end{equation}
For $0<r<\rho$, define
\begin{equation}\label{eq:eta-LK}
 \eta_{\gamma,a,\nu}(z)
 =-\gamma z+\frac{a}{2}z^2
 +\int_{\R}\left(e^{-zx}-1+zx\1_{\{|x|\le1\}}\right)\nu(dx).
\end{equation}

\begin{lemma}\label{lem:eta-Ar}
Under \eqref{eq:exp-moment}, $\eta_{\gamma,a,\nu}\in\Ar$ for every $0<r<\rho$.
\end{lemma}

\begin{proof}
For $n\ge2$, the coefficient contributed by the integral in \eqref{eq:eta-LK} is
$(-1)^n\int x^n\nu(dx)/n!$.  Tonelli's theorem and $r<\rho$ give
\[
 \sum_{n\ge2}\frac{r^n}{n!}\int_{\R}|x|^n\nu(dx)
 =\int_{\R}\bigl(e^{r|x|}-1-r|x|\bigr)\nu(dx)<\infty.
\]
Near the origin, the integrand is $O(x^2)$; outside $[-1,1]$, finiteness follows from \eqref{eq:exp-moment}.  The linear coefficient is also finite because the large-jump first moment is exponentially integrable.
\end{proof}

Let $\mu_s=\mu_{X(s)}$ be the moment series.  Independent stationary increments and Theorem~\ref{thm:classical} give
$\mu_{s+u}=\mu_s\hconv{-1}\mu_u$.

\begin{theorem}[L\'evy--Khintchine generator]\label{thm:LK-generator}
Assume \eqref{eq:exp-moment} and fix $0<r<\rho$.  Then
$M=\{\mu_s\}_{s\ge0}$ is a norm-continuous $\hconv{-1}$-semigroup in
$\mathcal A_{r,H}^{(-1)}$, and
\begin{equation}\label{eq:LK-generator}
 L_{r,H}^{(-1)}[M]
 =(\Phi_{-1}^H)^{-1}\circ
 M_{\eta_{\gamma,a,\nu}}\circ\Phi_{-1}^H.
\end{equation}
Moreover,
\begin{equation}\label{eq:LK-transform}
 \Phi_{-1}^H(\mu_s)(z)
 =\E[e^{-zX(s)}]
 =\exp\bigl(s\eta_{\gamma,a,\nu}(z)\bigr).
\end{equation}
\end{theorem}

\begin{proof}
The exponential-moment assumption implies $\E[e^{r|X(s)|}]<\infty$ for every $s\ge0$ and $r<\rho$, so the moment expansion may be summed absolutely on $|z|\le r$.  The equality on the left of \eqref{eq:LK-transform} then follows from
$\fall{-1}{n}=(-1)^n n!$.  The L\'evy--Khintchine formula gives the second equality.  Lemma~\ref{lem:eta-Ar} implies that the exponential in \eqref{eq:LK-transform} belongs to $\Ar$ and depends continuously on $s$ in the $\Ar$ norm.  Proposition~\ref{prop:Banach} and Theorem~\ref{thm:generator} now yield \eqref{eq:LK-generator}.
\end{proof}

\begin{corollary}[Transformed forward equation]\label{cor:forward}
For $f\in\mathcal A_{r,H}^{(-1)}$, set
$F_s=f\hconv{-1}\mu_s$.  Then
\begin{equation}\label{eq:forward}
 \frac{\partial}{\partial s}\Phi_{-1}^H(F_s)(z)
 =\eta_{\gamma,a,\nu}(z)\Phi_{-1}^H(F_s)(z),
 \qquad |z|\le r.
\end{equation}
\end{corollary}

\begin{proof}
By \eqref{eq:LK-transform},
$\Phi_{-1}^H(F_s)=\Phi_{-1}^H(f)e^{s\eta_{\gamma,a,\nu}}$ in $\Ar$.  Differentiate in the Banach-algebra norm.
\end{proof}

\begin{example}[Brownian motion]\label{ex:Brownian}
For standard Brownian motion, $(\gamma,a,\nu)=(0,1,0)$ and
$\eta(z)=z^2/2$.  Hence
\[
 L_{r,H}^{(-1)}
 =(\Phi_{-1}^H)^{-1}\circ M_{z^2/2}\circ\Phi_{-1}^H,
\]
and \eqref{eq:forward} becomes
\[
 \frac{\partial}{\partial s}\Phi_{-1}^H(F_s)(z)
 =\frac{z^2}{2}\Phi_{-1}^H(F_s)(z).
\]
\end{example}



\bibliographystyle{amsplain}
\bibliography{hurwitz_series_deformed_convolution}

\end{document}